\documentclass[12pt, reqno]{amsart}
\usepackage{amsmath, amsthm, amscd, amsfonts, amssymb, graphicx, color}
\usepackage[mathscr]{eucal}
\usepackage[bookmarksnumbered, colorlinks, plainpages]{hyperref}
\hypersetup{colorlinks=true,linkcolor=red, anchorcolor=green, citecolor=cyan, urlcolor=red, filecolor=magenta, pdftoolbar=true}

\newtheorem{theorem}{Theorem}

\newtheorem{proposition}[theorem]{Proposition}
\newtheorem{corollary}[theorem]{Corollary}
\theoremstyle{remark}
\newtheorem{remark}{Remark}
\newtheorem{example}{Example}

\DeclareMathAlphabet\mathoo{U}{eur}{b}{n}
\DeclareMathOperator{\co}{co}
\DeclareMathOperator{\Real}{Re}

\newcommand{\N}{N}

\usepackage{graphics}

\begin{document}
\setcounter{page}{1}

\title[An estimate of approximation]{An estimate of approximation \\of an analytic function of a matrix\\ by a rational function}

\author{M.~Ferus}
 \address{J.~Heyrovsk{\'{y}} Institute of Physical Chemistry, Academy of Sciences of the Czech Republic, Dolej\v{s}kova 3, 18223 Prague 8, Czech Republic}
\email{\textcolor[rgb]{0.00,0.00,0.84}{martin.ferus@jh-inst.cas.cz}}

\author{V.~G. Kurbatov$^*$}
 \address{Department of Mathematical Physics,
Voronezh State University\\ 1, Universi\-tet\-skaya Square, Voronezh 394018, Russia}
\email{\textcolor[rgb]{0.00,0.00,0.84}{kv51@inbox.ru}}
\thanks{$^*$ Corresponding author}

\author{I.~V. Kurbatova}
 \address{Department of Software Development and Information Systems Administration,
Vo\-ro\-nezh State University\\ 1, Universitetskaya Square, Voronezh 394018, Russia}
\email{\textcolor[rgb]{0.00,0.00,0.84}{irakurbatova@gmail.com}}

\subjclass{Primary 97N50; Secondary 65F60, 30E10, 41A20} 

\keywords{rational function of a matrix, matrix function, rational interpolation, error estimate, order reduction, Arnoldi method, Pad\'e approximant}

\date{\today}

\begin{abstract}
Let $A$ be a square complex matrix; $z_1$, \dots, $z_{\N}\in\mathbb C$~be arbitrary {\rm(}possibly repetitive{\rm)} points of interpolation; $f$ be an analytic function defined on a neighborhood of the convex hull of the union of the spectrum $\sigma(A)$ of the matrix $A$ and the points $z_1$, \dots, $z_{\N}$; and the rational function $r=\frac uv$ (with the degree of the numerator $u$ less than $N$) interpolates $f$ at these points {\rm(}counted according to their multiplicities{\rm)}. Under these assumptions estimates of the kind
\begin{equation*}
\bigl\lVert f(A)-r(A)\bigr\rVert\le
\max_{\substack{t\in[0,1]\\\mu\in\co\{z_1,z_{2},\dots,z_{\N}\}}}\biggl\lVert\Omega(A)[v(A)]^{-1}
\frac{\bigl(vf\bigr)^{{(\N)}}
\bigl((1-t)\mu\mathbf1+tA\bigr)}{\N!}\biggr\rVert,
\end{equation*}
where $\Omega(z)=\prod_{k=1}^N(z-z_k)$,
are proposed. As an example illustrating the accuracy of such estimates, an approximation of the impulse response of a dynamic system obtained using the reduced-order  Arnoldi method is considered, the actual accuracy of the approximation is compared with the estimate based on this paper.
\end{abstract}

\maketitle

\section*{Introduction}\label{s:Introduction}
It is well-known~\cite[ch.~IV,~\S~5]{Hartman73:eng}
that the solution of the initial value problem $x'(t)=Ax(t)$, $x(0)=x_0$, where $A$ is a square matrix, can be represented in the form $x(t)=e^{At}x_0$. Here $e^{At}$ is the result of the substitution of the matrix $A$ into the analytic function $\exp_t(\lambda)=e^{\lambda t}$. Some other analytic matrix functions arise in other applications~\cite{Bai-Demmel98,Burrage-Hale-Kay12,Byers-He-Mehrmann,Cardoso-Sadeghi18,
Grimm-Hochbruck08,Higham08,Jarlebring07,Kenney-Laub95,Schmelzer-Trefethen07}.

As a rule, an analytic function of a matrix can be calculated only approximately.
The usual way to approximately calculate an analytic function $f$ of a matrix $A$ is based on replacing $f$ with a polynomial or a rational function. The approximation by a rational function possesses some additional capabilities compared to a poly\-nomial one: it can be more accurate and can approximate an analytic function on an unbounded set. In this paper, we propose estimates (Theorem~\ref{t:Mathias:r4} and its corollaries) of $f(A)-r(A)$, where $r$ is a rational function that interpolates $f$. Similar estimates for polynomial approximation were described in~\cite{Kurbatov-Kurbatova-EMJ20,Mathias93a}.

As an application of these estimates, we consider the estimate of the accuracy of approximation of the impulse response of a dynamical system~\eqref{e:DS} using the Arnoldi reduced-order method (Theorems~\ref{t:est Arnoldi2} and~\ref{t:est Arnoldi1}). The rational (Sections~\ref{s:Arnoldi-2} and~\ref{s:Arnoldi-1}) reduced-order Arnoldi method is equivalent (Propostions~\ref{p:reduced order system} and~\ref{p:Arnoldi-r1}) to the approximation of the analytic function $\exp_t(\lambda)=e^{\lambda t}$ by a rational function $r_t$. However, the function $r_t$ is not calculated explicitly.

In Section~\ref{s:Numerical range} we illustrate (Examples~\ref{ex:1} and~\ref{ex:2}) our estimates of the Arnoldi reduced-order approximation using the properties of the numerical range.
In Section~\ref{s:experiment}, we describe a numerical experiment that shows the difference between our estimate and the actual approximation of the impulse response of a dynamical system obtained by means of the Arnoldi method.

In Sections~\ref{s:poly} and~\ref{s:rat}, we recall some facts connected with polynomial and rational interpolation. Section~\ref{s:Notation} contains general notation. In Section~\ref{s:DS}, we recall the general properties of reduced-order methods.

\section{Notation and other preliminaries}\label{s:Notation}

Let $n,m\in\mathbb N$.
We denote by $\mathbb C^{n\times m}$ the space of all complex $n\times m$-matrices. We denote the identity matrix by the symbol $\mathbf1$ or $\mathbf1_{n\times n}$. The symbol $A^H$ means the conjugate transpose of $A\in\mathbb C^{n\times m}$.
We represent elements $x\in\mathbb C^n$ as columns; thus the products $Ax$ and $y^HA$ make sense for $A\in\mathbb C^{n\times m}$, $x\in\mathbb C^m$, and $y\in\mathbb C^n$.
Usually, we identify a matrix $A\in\mathbb C^{n\times m}$ and the operator $x\mapsto Ax$ from $\mathbb C^m$ to $\mathbb C^n$ induced by $A$. In particular, by the image of a matrix we mean the image of the operator induced by it.

We assume that the domains of analytic functions under consideration are \emph{open} (maybe disconnected) subsets of $\mathbb C$.

Let $A\in\mathbb C^{n\times n}$. The spectrum of $A$ is the set $\sigma(A)=\{\lambda_1,\lambda_2,\dots,\lambda_m\}$ of all its eigenvalues. By the \emph{algebraic multiplicity} of $\lambda_j$ we mean the multiplicity of $\lambda_j$ as the root of the characteristic polynomial.

Let the domain of an analytic function $f$ contain the spectrum of a matrix $A\in\mathbb C^{n\times n}$. The \emph{function $f$ applied to the matrix} $A$ is the matrix
\begin{equation*}
f(A)=\frac1{2\pi i}\int_\Gamma f(\lambda)R_\lambda\,d\lambda,
\end{equation*}
where the contour $\Gamma$ encloses the spectrum of $A$ and
\begin{equation*}
R_\lambda=(\lambda\mathbf1-A)^{-1}
\end{equation*}
is the resolvent of $A$. The function
\begin{equation*}
\exp_t(\lambda)=e^{\lambda t}
\end{equation*}
is the most important example of the function $f$ from the point of view of applications.

We denote by $\co M$ the convex hull of a set $M\subset\mathbb C$.

\section{Polynomial interpolation}\label{s:poly}
Let $z_1,z_2,\dots,z_m\in \mathbb C$ be given distinct points called \emph{points {\rm(or} nodes{\rm)} of interpolation} and $n_1,n_2,\dots, n_m\in\mathbb N$ be their \emph{multiplicities}. We set
\begin{equation*}
N=\sum_{k=1}^mn_k.
\end{equation*}
Let $f$ be a function analytic in a neighborhood of the points of interpolation. The problem of \emph{polynomial interpolation} is~\cite{Gelfond:eng,Jordan,Walsh:eng} is to find a polynomial $p$ of degree $\le N-1$ satisfying the conditions
\begin{equation}\label{e:interpolation problem:poly:fold}
p^{(j)}(z_k)=f^{(j)}(z_k),\qquad k=1,\dots,m;\;j=0,1,\dots,n_k-1.
\end{equation}

\begin{proposition}[{\rm\cite[\S~3.1, Theorem 2]{Walsh:eng}}]\label{p:interpol is unique}
Interpolation problem~\eqref{e:interpolation problem:poly:fold} has a unique solution.
\end{proposition}

\begin{theorem}[{\rm\cite[p.~5]{Higham08}}]\label{t:f(A)=p(A)}
Let $A\in\mathbb C^{n\times n}$.
Let the spectrum $\sigma(A)$ of $A$ consists of the points $\lambda_1,\lambda_2,\dots,\lambda_m$, and let $w_1,w_2,\dots,w_{m}$ be their algebraic multiplicities.
Let the functions $f$ and $p$ be analytic in a neighborhood of $\sigma(A)$. Let the functions $f$ and $p$ and their derivatives coincide at $\lambda_i$ up to the order $w_i-1$:
\begin{equation*}
p^{(j)}(\lambda_k)=f^{(j)}(\lambda_k),\qquad k=1,2,\dots,m;\;j=0,1,\dots,w_m-1.
\end{equation*}
Then $f(A)=p(A)$.
\end{theorem}

\begin{theorem}[{\rm\cite[\S~3.1]{Walsh:eng}}]\label{t:f-p via Gamma}
Let $p$ be the interpolation polynomial satisfying~\eqref{e:interpolation problem:poly:fold} and a contour $\Gamma$ encloses the interpolation points $z_1$, $z_2$, \dots, $z_m$. Then at all points $z$ lying inside the contour $\Gamma$ one has
\begin{align}
p(z)&=\frac1{2\pi
i}\int_{\Gamma}\frac{\Omega(\lambda)-\Omega(z)}{\Omega(\lambda)(\lambda-z)}f(\lambda)\,d\lambda,\label{e:Walsh-Hermite:p} \\
f(z)-p(z)&=\Omega(z)\frac1{2\pi
i}\int_{\Gamma}\frac{f(\lambda)\,d\lambda}{\Omega(\lambda)(\lambda-z)},\label{e:Walsh-Hermite:f-p}
\end{align}
where
\begin{equation*}
\Omega(z)=\prod_{k=1}^m(z-z_k)^{n_k}.
\end{equation*}
\end{theorem}

Sometimes it is convenient to specify the multiplicities of points of interpolation implicitly.
Let $z_1, z_2,\dots,z_N\in\mathbb C$ be a list of points of interpolation (some of them may be repeated). We define the multiplicities of the points $z_1$, $z_2$, \dots, $z_N$ as the number of their repetition in this list.

Let a complex-valued function $f$ be defined and analytic in a neighborhood of the points $z_1$, $z_2$, \dots, $z_N$. The \emph{divided differences}
of the function $f$ with respect to the points $z_1$,
$z_2$, \dots, $z_N$ are defined~\cite{de_Boor05,Gelfond:eng,Jordan} by the recurrent relations
 \begin{equation}\label{e:divided differences}
 \begin{split}
f[z_i]&=f(z_i),\qquad\qquad\qquad\qquad\qquad\qquad\qquad 1\le i\le N,\\
f[z_i,z_{i+1}])&=\frac{f[z_{i+1}]-f[z_i]}{z_{i+1}-z_i},\qquad\qquad \qquad\qquad\qquad 1\le i\le N-1,\\
f[z_i,\dots,z_{i+m}]&=\frac{f[z_{i+1},\dots,z_{i+m}]
-f[z_{i},\dots,z_{i+m-1}]} {z_{i+m}-z_i},\quad 1\le i\le N-m.
 \end{split}
 \end{equation}
In these formulae, if the denominator vanishes, then the quotient is understood as the limit as $z_{i+m}-z_i\to0$; the limit always exists and coincides with the derivative with respect to one of the arguments of the previous divided difference.

\begin{proposition}[{\rm\cite[formula (52)]{de_Boor05}
}]\label{p:Gelfond(47)}
Let a function $f$ be analytic in a neighborhood of the convex hull of the points $z_1$, $z_2$, \dots,{ }$z_\N$ {\rm(}not necessarily different{\rm)}. Then
\begin{multline*}
f[z_1,z_{2},\dots,z_{\N}]=\int_0^1\int_0^{t_1}\dots\int_0^{t_{\N-2}}f^{{(\N-1)}}
\bigl(z_1+(z_2-z_1)t_1+\dots\\
\dots+(z_{\N-1}-z_{\N-2})t_{\N-2}+(z_{\N}-z_{\N-1})t_{\N-1}\bigr)\,dt_{\N-1}\dots dt_1.
\end{multline*}
\end{proposition}

\begin{theorem}[{\rm\cite[formula (51)]{de_Boor05}, \cite[formula (54)]{Gelfond:eng1}}]\label{t:f[] via Gamma}
Let a contour $\Gamma$ enclose the interpolation points $z_1$, $z_2$, \dots, $z_\N$ {\rm(}counted according to their multiplicities{\rm)} and the function $f$ be analytic in a neighborhood of the domain surrounded by $\Gamma$. Then
\begin{equation*}
f[z_1,z_{2},\dots,z_{\N}]=\frac1{2\pi i}\int_{\Gamma}\frac{f(\lambda)}{\Omega(\lambda)}\,d\lambda,
\end{equation*}
where
\begin{equation*}
\Omega(z)=\prod_{i=1}^\N(z-z_i).
\end{equation*}
\end{theorem}

\begin{proposition}[{\rm\cite[formula (52)]{Gelfond:eng1}}]\label{p:Lagrange remainder}
Let $f$ be an analytic function defined on an open set containing the interpolation points $z_1$, $z_2$, \dots, $z_\N$ {\rm(}counted according to their multiplicities{\rm)}. Let a polynomial $p$ of degree $\le N-1$ satisfy interpolation conditions~\eqref{e:interpolation problem:poly:fold}.
Then for all $z$ from the domain of definition of $f$ one has
\begin{equation*}
f(z)-p(z)=\Omega(z)\,f[z_1,z_2,\dots,z_\N,z].
\end{equation*}
\end{proposition}
\begin{proof}
The proof follows from Theorems~\ref{t:f-p via Gamma} and~\ref{t:f[] via Gamma}. \end{proof}

\section{Rational interpolation}\label{s:rat}
A \emph{rational function} is a function $r$ of a complex variable that can be represented in the form
\begin{equation*}
r(z)=\frac{u(z)}{v(z)}=\frac{a_0+a_1z+\dots+a_Lz^L}{b_0+b_1z+\dots+b_Mz^M},
\end{equation*}
where $u$ and $v$ are polynomials. We call the pair $[L/M]$ the \emph{degree} of $r$.

Let $z_1,z_2,\dots,z_m\in\mathbb C$ be given distinct points called \emph{points {\rm(or} nodes{\rm)} of interpolation} and $n_1,n_2,\dots, n_m\in\mathbb N$ be their \emph{multiplicities}. Let $f$ be an analytic function defined on a neighbourhood of the points of interpolation. The problem of \emph{rational interpolation} is~\cite{Baker-Graves-Morris96:eng,Walsh:eng} the problem of finding a rational function $r$ of degree $[L/M]$ or less satisfying the conditions
\begin{equation}\label{e:interpolation problem:rat:fold}
r^{(j)}(z_k)=f^{(j)}(z_k),\qquad k=1,\dots,m;\;j=0,1,\dots,n_k-1.
\end{equation}
Thus~\eqref{e:interpolation problem:rat:fold} consists of
\begin{equation*}
N=\sum_{k=1}^mn_k
\end{equation*}
conditions. Usually it is assumed that $L+M\le N-1$. It is also often assumed that the denominator $v$ is given. In the latter case, it is reasonable to assume that $L\le N-1$. If $v(z)\equiv1$, the problem of the rational interpolation is reduced to the \emph{polynomial} one.

\begin{proposition}\label{p:ratio interp}
Let the points of interpolation $z_1$, $z_2$, \dots, $z_m$ have multiplicities $n_1$ $n_2$, \dots, $n_m$. Let $u$, $v$, and $f$ be analytic functions defined on a neighborhood of the points $z_1$, $z_2$, \dots, $z_m${\rm;} $v(z_k)\neq0$, $k=1,2,\dots,m$. Then the interpolation conditions
\begin{equation}\label{e:inter con1}
\Bigl(\frac{u}{v}\Bigr)^{(j)}(z_k)=f^{(j)}(z_k),\qquad k=1,\dots,m;\;j=0,1,\dots,n_k-1,
\end{equation}
are equivalent to the interpolation conditions
\begin{equation}\label{e:inter con2}
u^{(j)}(z_k)=(vf)^{(j)}(z_k),\qquad k=1,\dots,m;\;j=0,1,\dots,n_k-1.
\end{equation}
\end{proposition}
\begin{proof}
Let conditions~\eqref{e:inter con1} be satisfied. Then for all $m=0,1,\dots,n_k-1$ we have (the argument $z_k$ is omitted for brevity)
\begin{equation*}
u^{(m)}=\Bigl(\frac uv\cdot v\Bigr)^{(m)}=\sum_{j=0}^m\binom{m}{j}\Bigl(\dfrac uv\Bigr)^{(j)}v^{(m-j)}
=\sum_{j=0}^m\binom{m}{j}(f)^{(j)}v^{(m-j)}=(vf)^{(m)}.
\end{equation*}

Conversely, let conditions~\eqref{e:inter con2} be satisfied. Then for all $m=0,1,\dots,n_k-1$ we have (the argument $z_k$ is again omitted for brevity)
\begin{align*}
\Bigl(\frac uv\Bigr)^{(m)}&=(uv^{-1})^{(m)}=\sum_{j=0}^m\binom{m}{j}u^{(j)}(v^{-1})^{(m-j)}
=\sum_{j=0}^m\binom{m}{j}(vf)^{(j)}(v^{-1})^{(m-j)}\\
&=\bigl[(vf)v^{-1}\bigr]^{(m)}=f^{(m)}.\qed
\end{align*}
\renewcommand\qed{}
\end{proof}

\begin{corollary}\label{c:Walsh}
Let the points of interpolation $z_1$, $z_2$, \dots, $z_m$ have multiplicities $n_1$ $n_2$, \dots, $n_m$. Let $N=\sum_{k=1}^mn_k$. Let $f$ be an analytic function defined on a neighborhood of the points $z_1$, $z_2$, \dots, $z_m$.
Let $v$ be a given polynomial such that $v(z_k)\neq0$, $k=1,2,\dots,m$. Then there exists a unique polynomial $u$ of degree $L\le N-1$ such that the rational function $r=\frac uv$ satisfies interpolation conditions~\eqref{e:interpolation problem:rat:fold}.
\end{corollary}
\begin{proof}
By Proposition~\ref{p:ratio interp}, it is enough to show that there exists a poly\-nomial $u$ that interpolates the function $vf$. By Proposition~\ref{p:interpol is unique}, this problem has a unique solution.
\end{proof}

\begin{proposition}\label{p:Lagrange remainder:r}
Let $f$ be an analytic function defined on an open set $U$ containing the points of interpolation $z_1$, $z_2$, \dots, $z_\N$ {\rm(}not necessarily different{\rm)}.
Let a rational function $r=\frac uv$ of degree $[L/M]$ satisfy interpolation conditions\footnote{Note that to check (5), one should first calculate the multiplicities of the interpolation points.}~\eqref{e:interpolation problem:rat:fold}, $L\le N-1$, and $v(z_k)\neq0$, $k=1,2,\dots,N$. Then for all $z\in U$ such that $v(z)\neq0$ one has
\begin{equation*}
f(z)-r(z)=\frac{\Omega(z)}{v(z)}\,\bigl(vf\bigr)[z_1,z_2,\dots,z_\N,z],
\end{equation*}
where 
\begin{equation*}
\Omega(z)=\prod_{k=1}^N(z-z_k).
\end{equation*}
\end{proposition}
\begin{proof} 
By Proposition~\ref{p:ratio interp}, the polynomial $u$ interpolates the function $vf$. Therefore, by Proposition~\ref{p:Lagrange remainder},
\begin{equation*}
v(z)f(z)-u(z)=\Omega(z)\bigl(vf\bigr)[z,z_1,z_2,\dots,z_\N,z].
\end{equation*}
Hence
\begin{equation*}
f(z)-\frac{u(z)}{v(z)}=\frac{\Omega(z)}{v(z)}\bigl(vf\bigr)[z,z_1,z_2,\dots,z_\N,z].\qed
\end{equation*}
\renewcommand\qed{}
\end{proof}

\section{The estimate}\label{s:estimate}
In this section, we present our estimate and its variants.

\begin{theorem}\label{t:Mathias:r1}
Let $A\in\mathbb C^{n\times n}${\rm;} $z_1$, $z_2$, \dots, $z_{\N}\in\mathbb C$~be arbitrary {\rm(}possibly repetitive{\rm)} points of interpolation{\rm;} $f$ be an analytic function defined on a neighbor\-hood of the convex hull of the union of the spectrum $\sigma(A)$ of the matrix $A$ and the points $z_1$, $z_2$, \dots, $z_{\N}${\rm;} a rational function $r=\frac uv$ of degree $[L/M]$ satisfy interpolation conditions~\eqref{e:interpolation problem:rat:fold}{\rm;} $L\le N-1${\rm;}  $v(z_k)\neq0$, $k=1,2,\dots,N$, and $v(\lambda)\neq0$ for $\lambda\in\sigma(A)$. Then
\begin{multline*}
f(A)-r(A)=\Omega(A)[v(A)]^{-1}\int_0^1\int_0^{t_1}\dots\int_0^{t_{\N-1}}\bigl(vf\bigr)^{{(\N)}}
\bigl((1-t_1)z_1\mathbf1\\
+(t_1-t_2)z_2\mathbf1+
\dots+(t_{\N-1}-t_{\N})z_{\N}\mathbf1+t_{\N}A\bigr)\,dt_{\N}dt_{\N-1}\dots dt_1,
\end{multline*}
where
\begin{equation*}
\Omega(z)=\prod_{k=1}^\N(z-z_k).
\end{equation*}
\end{theorem}

\begin{proof}
By Proposition~\ref{p:Lagrange remainder:r},
\begin{equation*}
f(z)-r(z)=\frac{\Omega(z)}{v(z)}\,\bigl(vf\bigr)[z_1,z_2,\dots,z_\N,z].
\end{equation*}
On the other hand, by Proposition~\ref{p:Gelfond(47)},
\begin{multline*}
\bigl(vf\bigr)[z_1,z_{2},\dots,z_{\N},z]=\int_0^1\int_0^{t_1}\dots\int_0^{t_{\N-1}}\bigl(vf\bigr)^{{(\N)}}
\bigl(z_1+(z_2-z_1)t_1+\dots\\
+(z_{\N}-z_{\N-1})t_{\N-1}+(z-z_{\N})t_{\N}\bigr)\,dt_{\N}dt_{\N-1}\dots dt_1.
\end{multline*}
Or
\begin{multline*}
\bigl(vf\bigr)[z_1,z_{2},\dots,z_{\N},z]=\int_0^1\int_0^{t_1}\dots\int_0^{t_{\N-1}}\bigl(vf\bigr)^{{(\N)}}
\bigl((1-t_1)z_1\\
+(t_1-t_2)z_2+\dots+(t_{\N-1}-t_{\N})z_{\N}+t_{\N}z\bigr)\,dt_{\N}dt_{\N-1}\dots dt_1.
\end{multline*}
Therefore
\begin{multline*}
f(A)-r(A)=\Omega(A)[v(A)]^{-1}\int_0^1\int_0^{t_1}\dots\int_0^{t_{\N-1}}\bigl(vf\bigr)^{{(\N)}}
\bigl((1-t_1)z_1\mathbf1\\
+(t_1-t_2)z_2\mathbf1+
\dots+(t_{\N-1}-t_{\N})z_{\N}\mathbf1+t_{\N}A\bigr)\,dt_{\N}dt_{\N-1}\dots dt_1.\qed
\end{multline*}
\renewcommand\qed{}
\end{proof}

\begin{theorem}\label{t:Mathias:r4}
Under assumptions of Theorem~\ref{t:Mathias:r1} for any linear functional $\xi$ on the linear space $\mathbb C^{n\times n}$ of matrices, one has
\begin{equation*}
\bigl|\xi\bigl[f(A)-r(A)\bigr]\bigr|\le
\max_{\substack{t\in[0,1]\\\mu\in\co\{z_1,z_{2},\dots,z_{\N}\}}}\biggl|\xi\biggl[\Omega(A)[v(A)]^{-1}
\frac{\bigl(vf\bigr)^{{(\N)}}\bigl((1-t)\mu\mathbf1+tA\bigr)}{\N!}\biggr]\biggr|.\qed
\end{equation*}
\end{theorem}

\begin{remark}\label{r:partial fraction}
The matrix $v(A)$ (as any other polynomial of a matrix) is often badly conditioned. Therefore, a direct calculation of $[v(A)]^{-1}$ can be numerically unstable. To overcome this problem, one can first calculate the partial fraction decomposition of $\lambda\mapsto\Omega(\lambda)/v(\lambda)$ and then substitute $A$ in it.
\end{remark}

\begin{proof}
From Theorem~\ref{t:Mathias:r1} it follows that
\begin{equation}\label{e:long xi}
\begin{split}
\bigl|\xi\bigl[f(A)&-r(A)\bigr]\bigr|
=\biggl|\xi\biggl(\int_0^1\int_0^{t_1}\dots\int_0^{t_{\N-1}}\Omega(A)[v(A)]^{-1}\bigl(vf\bigr)^{{(\N)}}
\bigl((1-t_1)z_1\mathbf1+\dots\\
&+(t_{\N-1}-t_{\N})z_{\N}\mathbf1+t_{\N}A\bigr)\,dt_{\N}dt_{\N-1}\dots dt_1\biggr)\biggr|\\
&=\biggl|\int_0^1\int_0^{t_1}\dots\int_0^{t_{\N-1}}\xi\bigl[\Omega(A)[v(A)]^{-1}\bigl(vf\bigr)^{{(\N)}}
\bigl((1-t_1)z_1\mathbf1+\dots\\
&+(t_{\N-1}-t_{\N})z_{\N}\mathbf1+t_{\N}A\bigr)\bigr]\,dt_{\N}dt_{\N-1}\dots dt_1\biggr|\\
&\le\int_0^1\int_0^{t_1}\dots\int_0^{t_{\N-1}}
\max_{t_1,\dots,t_N}\Bigl|\xi\Bigl[\Omega(A)[v(A)]^{-1}\bigl(vf\bigr)^{{(\N)}}
\bigl((1-t_1)z_1\mathbf1+\dots\\
&+(t_{\N-1}-t_{\N})z_{\N}\mathbf1+t_{\N}A\bigr)\Bigr]\Bigr|\,dt_{\N}dt_{\N-1}\dots dt_1.
\end{split}
\end{equation}

It is easy to see that the complex number
\begin{equation*}
\frac1{1-t_{\N}}\bigl((1-t_1)z_1+(t_1-t_2)z_2+\dots+(t_{\N-1}-t_{\N})z_{\N}\bigr)
\end{equation*}
runs over the convex hull $\co\{z_1,z_{2},\dots,z_{\N}\}$, and
\begin{equation*}
\int_0^1\int_0^{t_1}\dots\int_0^{t_{\N-1}}\,dt_{\N}\dots dt_1=\frac1{\N!}.
\end{equation*}
Therefore estimate~\eqref{e:long xi} implies that
\begin{equation*}
\bigl|\xi\bigl[f(A)-r(A)\bigr]\bigr|\le
\max_{\substack{t\in[0,1]\\\mu\in\co\{z_1,z_{2},\dots,z_{\N}\}}}\biggl|\xi\biggl[\Omega(A)[v(A)]^{-1}
\frac{\bigl(vf\bigr)^{{(\N)}}\bigl((1-t)\mu\mathbf1+tA\bigr)}{\N!}\biggr]\biggr|.\qed
\end{equation*}
\renewcommand\qed{}
\end{proof}

\begin{corollary}\label{c:Mathias:r5}
Under assumptions of Theorem~\ref{t:Mathias:r1} for any $b,d\in\mathbb C^n$,
\begin{multline*}
\bigl|d^H(f(A)-r(A))b\bigr|\\
\le\max_{\substack{t\in[0,1]\\\mu\in\co\{z_1,z_{2},\dots,z_{\N}\}}}\biggl|d^H\biggl[\Omega(A)[v(A)]^{-1}
\frac{\bigl(vf\bigr)^{{(\N)}}
\bigl((1-t)\mu\mathbf1+tA\bigr)}{\N!}\biggr]b\biggr|.
\end{multline*}
\end{corollary}
\begin{proof}
It suffices to observe that the rule $A\mapsto d^HAb$ is a linear functional on the space of matrices and refer to Theorem~\ref{t:Mathias:r4}.
\end{proof}

\begin{corollary}\label{c:Mathias:r6}
Under assumptions of Theorem~\ref{t:Mathias:r1}  for any $b\in\mathbb C^n$ {\rm(}and the Euclidian norm $\lVert\cdot\rVert_2$ on $\mathbb C^n${\rm)},
\begin{equation*}
\bigl\lVert\bigl(f(A)-r(A)\bigr)b\bigr\rVert_2\le
\max_{\substack{t\in[0,1]\\\mu\in\co\{z_1,z_{2},\dots,z_{\N}\}}}
\biggl\lVert\Omega(A)[v(A)]^{-1}
\frac{\bigl(vf\bigr)^{{(\N)}}
\bigl((1-t)\mu\mathbf1+tA\bigr)}{\N!}b\biggr\rVert_2.
\end{equation*}
\end{corollary}
\begin{proof}
If $\bigl(f(A)-r(A)\bigr)b=0$ the proof is evident.
If $\bigl(f(A)-r(A)\bigr)b\neq0$, we set $d=\bigl(f(A)-r(A)\bigr)b\big/\lVert \bigl(f(A)-r(A)\bigr)b\rVert_2$. After that we refer to Corollary~\ref{c:Mathias:r5}.
\end{proof}

\begin{corollary}\label{c:Mathias:r2}
Under assumptions of Theorem~\ref{t:Mathias:r1} {\rm(}for any norm on the space of matrices{\rm)},
\begin{equation*}
\bigl\lVert f(A)-r(A)\bigr\rVert\le
\max_{\substack{t\in[0,1]\\\mu\in\co\{z_1,z_{2},\dots,z_{\N}\}}}\biggl\lVert\Omega(A)[v(A)]^{-1}
\frac{\bigl(vf\bigr)^{{(\N)}}
\bigl((1-t)\mu\mathbf1+tA\bigr)}{\N!}\biggr\rVert.
\end{equation*}
\end{corollary}

\begin{proof}
Let $\xi$~be a linear functional on the space of matrices (equipped by an arbitrary norm) such that $\lVert\xi\rVert=1$ and
\begin{equation*}
\lVert f(A)-r(A)\rVert=\xi\bigl(f(A)-r(A)\bigr).
\end{equation*}
Such a functional exists by the Hahn--Banach theorem~\cite[Theorem 2.7.4]{Hille-Phillips:eng}. Then from Theorem~\ref{t:Mathias:r4} we have
\begin{align*}
\lVert f(A)&-r(A)\rVert
=\xi\bigl[f(A)-r(A)\bigr]=\bigl|\xi\bigl[f(A)-r(A)\bigr]\bigr|\\
&\le
\max_{\substack{t\in[0,1]\\\mu\in\co\{z_1,z_{2},\dots,z_{\N}\}}}\Bigl|\xi\bigl[\Omega(A)[v(A)]^{-1}
\frac{\bigl(vf\bigr)^{{(\N)}}\bigl((1-t)\mu\mathbf1+tA\bigr)}{\N!}\bigr]\Bigr|\\
&\le
\max_{\substack{t\in[0,1]\\\mu\in\co\{z_1,z_{2},\dots,z_{\N}\}}}\Bigl\lVert\Omega(A)[v(A)]^{-1}
\frac{\bigl(vf\bigr)^{{(\N)}}\bigl((1-t)\mu\mathbf1+tA\bigr)}{\N!}\Bigr\rVert.\qed
\end{align*}
\renewcommand\qed{}
\end{proof}

\begin{corollary}\label{c:Mathias:r3}
Let a function $f$ be analytic on an open circle of radius $r$ centered at a point $z_0$ and the spectrum of a square matrix $A$ be contained in this circle.
Then the difference between the exact value $f(A)$ and the Pad\'e approximant $r=\frac uv$ of degree $[L/M]$ of the function $f$ at the point $z_0$ applied to $A$ admits the estimate
\begin{equation*}
\bigl\Vert f(A)-r(A)\bigr\Vert\le
\max_{t\in[0,1]}\Bigl\lVert (A-z_0\mathbf1)^N[v(A)]^{-1}
\frac{\bigl(vf\bigr)^{(N)}\bigl((1-t)z_0\mathbf1+tA\bigr)}{\N!}\Bigr\rVert,
\end{equation*}
where $N=L+M+1$. It is assumed that $v(\lambda)\neq0$ for $\lambda\in\sigma(A)$.
\end{corollary}
\begin{proof}
It suffices to recall that the Pad\'e approximant is an inter\-po\-la\-tion rational function that corresponds to a single interpolation point $z_0$ of multiplicity~$N$.
\end{proof}
An analogue of Corollary~\ref{c:Mathias:r3} for approximation by the Taylor polynomials is established in~\cite{Mathias93a}.

\section{Reduced-order methods}\label{s:DS}
In this Section, we describe an application of Corollary~\ref{c:Mathias:r5} for accuracy of approximation of the impulse response of a single-input, single-output dynamical system~\cite{Antoulas} based on the Arnoldi type method of order reduction.

We consider a dynamical system~\cite{Antoulas,Polderman-Willems} with the input $u$ and the output $y$ governed by the equations
\begin{equation}\label{e:DS}
\begin{split}
x'(t)&=Ax(t)+bu(t),\\
y(t)&=d^Hx(t),
\end{split}
\end{equation}
where $A\in\mathbb C^{n\times n}$ and $b,d\in\mathbb C^n$ are given matrices.
The following fact is well known.
\begin{theorem}[{\rm\cite[p.~65]{Antoulas},~\cite[p.~46]{Zhou-Doyle-Glover96}}]\label{t:sol of DS}
The solution of problem~\eqref{e:DS} satisfying the initial condition
\begin{equation*}
x(t_0)=x_0
\end{equation*}
can be represented as
\begin{equation*}
y(t)=d^H\biggl(\exp_{t-t_0}(A)x_0 + \int_{t_0}^{t} \exp_{t-r}(A)bu(r)\,dr\biggr),\qquad t\ge t_0,
\end{equation*}
where $\exp_t(\lambda)=e^{\lambda t}$.
\end{theorem}
This formula shows that the principal part of solving the problem~\eqref{e:DS} consists in finding the function $t\mapsto d^H\exp_t(A)b$. We call the function $t\mapsto d^H\exp_t(A)b$ the (\emph{scalar{\rm)} impulse response} and we call the function $t\mapsto \exp_t(A)b$ the (\emph{vector{\rm)} impulse response}.

A system of \emph{reduced order} with respect to~\eqref{e:DS} is~\cite{Antoulas,Grimme97,Vorst} the system governed by the equations
 \begin{equation}\label{e:dynamic system'}
 \begin{split}
\hat x'(t)&=\widehat A\hat x(t)+\hat bu(t),\\
\hat y(t)&=\hat d^H\hat x(t),
 \end{split}
 \end{equation}
in which the order $\hat n$ of the matrix $\widehat A$ is substantially less than the order $n$ of the matrix $A$, but the output $\hat y$ is close to the output $y$ of problem~\eqref{e:DS}.

We say that problem~\eqref{e:dynamic system'} is constructed by a \emph{projection} method if the coefficients $\widehat A,\hat b,\hat d$ in~\eqref{e:dynamic system'} are expressed in terms of the coefficients of initial problem~\eqref{e:DS} by the formulae
 \begin{equation}\label{e:hats}
\widehat A=\Lambda AV,\qquad
\hat b=\Lambda b,\qquad
\hat d=Vd,
 \end{equation}
where $V\in\mathbb C^{n\times\hat n}$ and $\Lambda\in\mathbb C^{\hat n\times n}$ are some matrices.

We will always assume that the following \emph{normalizing} assumption is fulfilled:
   \begin{equation}\label{e:Labmda V=1}
 \begin{split}
\Lambda V&=\mathbf1_{\hat n\times\hat n},
 \end{split}
 \end{equation}
where $\mathbf1_{\hat n\times\hat n}$ is the identity matrix of the size $\hat n\times\hat n$. Moreover, usually we will assume that condition~\eqref{e:Lambda=VH} from the following proposition is fulfilled.

 \begin{proposition}\label{p:S}
Let $S\in\mathbb C^{\hat n\times\hat n}$ be an arbitrary invertible matrix. We set $V_1=VS$,
$\Lambda_1=S^{-1}\Lambda$,
\begin{equation*}
\widehat A_1=\Lambda_1 A V_1, \qquad\hat b_1=\Lambda_1b,\qquad\hat d^H_1=d^HV_1.
\end{equation*}
Then the solution $\hat y$ of the problem
\begin{equation}\label{e:S-reduced}
\begin{split}
{\hat x}_1'&=\widehat{A}_1{\hat x}_1+\hat b_1u(t),\\
\hat y(t)&=\hat d^H_1{\hat x}(t)
 \end{split}
\end{equation}
coincides with the solution $\hat y$ of problem~\eqref{e:dynamic system'}.
\end{proposition}
 \begin{proof}
We make the change ${\hat x}=S{\hat x}_1$ in problem~\eqref{e:dynamic system'}:
 \begin{equation*}
 \begin{split}
S{\hat x}_1'(t)&=\Lambda AVS{\hat x}_1(t)+\hat bu(t),\\
\hat y(t)&=\hat d^HS{\hat x}_1(t).
 \end{split}
 \end{equation*}
We multiply the differential equation by $S^{-1}$ and use the equality $S^{-1}S=\mathbf1$:
 \begin{equation*}
 \begin{split}
{\hat x}_1'(t)&=S^{-1}\Lambda AVS{\hat x}_1(t)+S^{-1}\hat bu(t),\\
\hat y(t)&=\hat d^HS{\hat x}_1(t).
 \end{split}
 \end{equation*}
We rewrite these equations as
 \begin{equation*}
 \begin{split}
{\hat x}_1'(t)&=\Lambda_1 AV_1{\hat x}_1(t)+S^{-1}\Lambda bu(t),\\
\hat y(t)&=d^HVS{\hat x}_1(t),
 \end{split}
 \end{equation*}
We have arrived at system~\eqref{e:S-reduced}.
 \end{proof}

In connection with Proposition~\ref{p:S}, the columns of the matrix $V$ and the rows of the matrix $\Lambda$ are usually taken orthonormal. This leads to the fact that calculations by the formula $\widehat A=\Lambda AV$ result in minimal round-off errors.

\begin{proposition}\label{p:Lambda=VH}
Let the columns of the matrix $V$ be orthonormalized and the matrix $\Lambda$ be defined by the formula
\begin{equation}\label{e:Lambda=VH}
\Lambda=V^{H}.
\end{equation}
Then assumption~\eqref{e:Labmda V=1} is fulfilled, and the matrix $V\Lambda\in\mathbb C^{n\times n}$ defines an orthogonal projector $P$ onto the linear span of the columns of the matrix $V$.
\end{proposition}

\begin{proof}
By~\eqref{e:Lambda=VH}, the matrix $\Lambda V$ is the Gram matrix of the columns of the matrix $V$. This observation implies the first statement.

We extend the set consisting of $\hat n$ columns of the matrix $V$ to an orthonormal basis of $\mathbb C^n$. We take an arbitrary vector $x\in\mathbb C^n$. By~\eqref{e:Lambda=VH}, the vector $\Lambda x$ consists of the first $n$ coordinates of $x$ in this basis. Therefore, the vector $V(\Lambda x)\in\mathbb C^n$ coincides with the projection of $x$ onto the linear span of the first $\hat n$ basis vectors.
\end{proof}

\begin{corollary}\label{c:Lambda=VH}
Under assumptions of Proposition~\ref{p:Lambda=VH} $AV-V\widehat{A}=(\mathbf1-P)AV$.
\end{corollary}
\begin{proof}
Indeed, $AV-V\widehat{A}=AV-V\Lambda AV=(\mathbf1-P)AV$.
\end{proof}

It is clear that the fundamental part in the construction of reduced-order model~\eqref{e:hats} is the choice of matrices $V$ and $\Lambda$. Proposition~\ref{p:Knizhnerman} below shows that the solution $\hat y$ of the reduced-order problem~\eqref{e:Labmda V=1} is determined by the linear span of the columns of the matrices $V$ and $\Lambda^H$.

\section{Two-sided rational Arnoldi}\label{s:Arnoldi-2}

We consider two variants of the Arnoldi method~\cite{Antoulas,Grimme97,Higham08,Lee-Chu-Feng06,Simoncini-Szyld07,Vorst,
Voevodin-Kuznetsov:rus-eng} 
of order reduction. We always assume that assumption~\eqref{e:Lambda=VH} is fulfilled.

Let $\varkappa_0$ and $\chi_{0}$ be given nonnegative integers called \emph{multiplicities}.
Let the image of the operator $V$ contains the vectors
\begin{equation}\label{e:Krylov vectors:V}
b,Ab,A^{2}b,\dots,A^{\varkappa_0-1}b,
\end{equation}
and the image of the operator $\Lambda^H$ contains the vectors
\begin{equation}\label{e:Krylov vectors:Lambda}
d,A^Hd,(A^H)^2d,\dots,(A^H)^{\chi_{0}-1}d.
\end{equation}
Further, let $\lambda_{1}, \lambda_{2}, \dots,\lambda_{m}\in\mathbb{C}$ be points not lying in the spectrum of $A$, and $\varkappa_{1}, \dots, \varkappa_{m}$ and $\chi_1,\dots,\chi_m$ be nonnegative integers. We additionally assume that
the image of the operator $V$ contains the vectors
\begin{equation}\label{e:more Krylov vectors:V}
 \begin{split}
(\lambda_{1}I-A)^{-1}b,\;&(\lambda_{1}I-A)^{-2}b,\;\dots,\;(\lambda_{1}I-A)^{-\varkappa_{1}}b,\\
\dots\dots\dots\dots\dots&\dots\dots\dots\dots\dots\dots\dots\dots\dots\dots\dots\\
(\lambda_{m}I-A)^{-1}b,\;&(\lambda_{m}I-A)^{-2}b,\;\dots,\;(\lambda_{m}I-A)^{-\varkappa_{m}}b
 \end{split}
\end{equation}
and the image of the operator $\Lambda^H$ contains the vectors
\begin{equation}\label{e:more Krylov vectors:Lambda}
 \begin{split}
(\bar\lambda_{1}I-A^H)^{-1}d,\;&(\bar\lambda_{1}I-A^H)^{-2}d,\;\dots,\;(\bar\lambda_{1}I-A^H)^{-\chi_{1}}d,\\
\dots\dots\dots\dots\dots&\dots\dots\dots\dots\dots\dots\dots\dots\dots\dots\dots\dots\\
(\bar\lambda_{m}I-A^H)^{-1}d,\;&(\bar\lambda_{m}I-A^H)^{-2}d,\;\dots,\;(\bar\lambda_{m}I-A^H)^{-\chi_{m}}d.
 \end{split}
\end{equation}
It is convenient to interpret vectors~\eqref{e:Krylov vectors:V} and~\eqref{e:Krylov vectors:Lambda} as analogues of vectors~\eqref{e:more Krylov vectors:V} and~\eqref{e:more Krylov vectors:Lambda} corresponding to the point $\lambda_{0}=\infty$.

In \emph{two-sided Arnoldi} methods, it is assumed that the image of $V=\Lambda^H$ is defined as the linear span of vectors~\eqref{e:Krylov vectors:V}--\eqref{e:more Krylov vectors:Lambda}. In \emph{one-sided Arnoldi} methods, it is assumed that the image of $V=\Lambda^H$ is defined as the linear span of vectors~\eqref{e:Krylov vectors:V} and~\eqref{e:more Krylov vectors:V} only. We assume that these vectors are linear independent. It is convenient to combine the verification of the linear independence with the orthonormal process. The columns of the matrix $V=\Lambda^H$ are usually taken orthonormal.

By Proposition~\ref{p:S}, reduced-order system~\eqref{e:dynamic system'} is defined by the points $\lambda_{0}=\infty$, $\lambda_{1} , \dots, \lambda_{m}\in\mathbb C$ and their multiplicities $\varkappa_{k}$ and $\chi_{k}$, $k=0,\dots,m$. The quality of approximation of system~\eqref{e:DS} by system~\eqref{e:dynamic system'} depends only of these parameters.

\begin{proposition}[{\rm\cite[Lemma 3.1]{Druskin-Knizhnerman-Zaslavsky09}, \cite[Lemma 3.1]{Guttel13}}]\label{p:Knizhnerman}

\noindent\begin{itemize}
 \item[{\rm(a)}]
Let $V\in\mathbb C^{n\times\hat n}$ and $\Lambda\in\mathbb C^{\hat n\times n}$ satisfy assumption~\eqref{e:Labmda V=1}.
Let the image of the matrix $V$ contain vectors~\eqref{e:Krylov vectors:V} and~\eqref{e:more Krylov vectors:V}, and the image of the matrix $\Lambda^H$ contain vectors~\eqref{e:Krylov vectors:Lambda} and~\eqref{e:more Krylov vectors:Lambda}.
We consider matrices~\eqref{e:hats}.
Let points $\lambda_1$, \dots, $\lambda_m\in\mathbb C$ be not both in the spectrum of $A$ and the spectrum of $\widehat{A}$.
Then for any rational function $r$ of the form
\begin{equation*}
r(\lambda)=\sum_{k=1}^{m}\sum_{j=1}^{\varkappa_k+\chi_k}\frac{g_{jk}}{{(\lambda_k-\lambda)^{j}}}+
\sum_{j=0}^{\varkappa_0+\chi_0-1}g_{j0}\lambda^j
\end{equation*}
one has
\begin{equation*}
d^Hr(A)\,b=\hat d^Hr(\widehat{A})\,\hat b.
\end{equation*}
 \item[{\rm(b)}]
Let $V\in\mathbb C^{n\times\hat n}$ satisfy assumption $V^HV=\mathbf1_{\hat n\times\hat n}$.
Let the image of the matrix $V$ contain vectors~\eqref{e:Krylov vectors:V} and~\eqref{e:more Krylov vectors:V}.
We consider matrices~\eqref{e:hats} with $\Lambda=V^H$.
Let points $\lambda_1$, \dots, $\lambda_m\in\mathbb C$ be not both in the spectrum of $A$ and the spectrum of $\widehat{A}$.
Then for any rational function $r$ of the form
 \begin{equation*}
r(\lambda)=\sum_{k=1}^{m}\sum_{j=1}^{\varkappa_k}\frac{g_{jk}}{{(\lambda_k-\lambda)^{j}}}+
\sum_{j=0}^{\varkappa_0-1}g_{j0}\lambda^j
\end{equation*}
one has
\begin{equation*}
r(A)b=Vr(\widehat{A})\hat b.
\end{equation*}
\end{itemize}
\end{proposition}

 \begin{proposition}[{\rm\cite{Beckermann-Reichel09}}]\label{p:reduced order system}
Let the image of the matrix $V$ contain vectors~\eqref{e:Krylov vectors:V} and~\eqref{e:more Krylov vectors:V} and the image of the matrix $\Lambda^H$ contain vectors~\eqref{e:Krylov vectors:Lambda} and~\eqref{e:more Krylov vectors:Lambda}.
Let $V\in\mathbb C^{n\times\hat n}$ and $\Lambda\in\mathbb C^{\hat n\times n}$ satisfy assumption~\eqref{e:Labmda V=1}.
We consider matrices~\eqref{e:hats}.
Let $\sigma(\widehat A)$ consist of the points $\hat\mu_1,\dots,\hat\mu_{\hat m}\in\mathbb C$, and let $\hat w_1,\dots,\hat w_{\hat m}$ be their algebraic multiplicities.

Let the points $\lambda_1$, \dots, $\lambda_m\in\mathbb C$ be not both in the spectrum of $A$ and the spectrum of $\widehat{A}$.
Let a rational function\footnote{It may happen that the number of coefficients $\sum_{k=0}^m\varkappa_k+\chi_k$ in formula~\eqref{e:r_t} is less than the number $\sum_{k=1}^{\hat m}\hat w_k$ of interpolation conditions.} $r_t$ of the form
\begin{equation}\label{e:r_t}
r_t(\lambda)=\sum_{k=1}^{m}\sum_{j=1}^{\varkappa_k+\chi_k}\frac{g_{jk}(t)}{{(\lambda_k-\lambda)^{j}}}+
\sum_{j=0}^{\varkappa_0+\chi_0-1}g_{j0}(t)\lambda^j
\end{equation}
satisfy the following interpolation assumptions{\rm:} the function
$r_t$ coincides with the function $\exp_t(\lambda)=e^{\lambda t}$ at the points $\hat\mu_1,\dots,\hat\mu_{\hat m}$ with the derivatives up to the orders $\hat w_1-1,\dots,\hat w_{\hat m}-1${\rm:}
\begin{equation*}
r_t^{(j)}(\hat\mu_k)=\exp_t^{(j)}(\hat\mu_k),\qquad k=1,2,\dots,\hat m;\;j=0,1,\dots,w_{\hat m}-1.
\end{equation*}

Then one has
\begin{equation*}
d^Hr_t(A)b=\hat d^H\exp_t(\widehat{A})\hat b,\qquad t\in\mathbb R.
\end{equation*}
 \end{proposition}
\begin{proof}
Since the function $r_t$ satisfies the interpolation conditions, by Theorem~\ref{t:f(A)=p(A)}, we have
$\exp_t(\widehat{A})=r_t(\widehat{A})$. Now from Proposition~\ref{p:Knizhnerman}(a) it follows that $d^Hr_t(A)b=\hat d^Hr_t(\widehat{A})\hat b$.
\end{proof}

\begin{theorem}\label{t:est Arnoldi2}
Let vectors~\eqref{e:Krylov vectors:V}, \eqref{e:Krylov vectors:Lambda}, \eqref{e:more Krylov vectors:V}, and~\eqref{e:more Krylov vectors:Lambda} form a basis\footnote{For example, if the matrix $A$ is Hermitian, and $b=d$ and $\varkappa_0=\chi_0$, then vectors~\eqref{e:Krylov vectors:V} coincide with vectors \eqref{e:Krylov vectors:Lambda} and the linear independence does not hold. Nevertheless, if vectors~\eqref{e:Krylov vectors:V}, \eqref{e:Krylov vectors:Lambda}, \eqref{e:more Krylov vectors:V}, and~\eqref{e:more Krylov vectors:Lambda} are calculated successively, one can easily exclude linear dependent vectors.} in the image of the matrix $V$, and let $V^HV=\mathbf1_{\hat n\times\hat n}$.
Consider matrices~\eqref{e:hats} with $\Lambda=V^H$.
Let $\sigma(\widehat A)$ consists of the points $\hat\mu_1,\dots,\hat\mu_{\hat m}\in\mathbb C$, and let $\hat w_1,\dots,\hat w_{\hat m}$ be their algebraic multiplicities.

Let the points $\lambda_1$, \dots, $\lambda_m\in\mathbb C$ be not both in $\sigma(A)$ and in $\sigma(\widehat{A})$.
Let the reduced-order system be defined by~\eqref{e:dynamic system'}.
Then the difference between scalar impulse respon\-ses of initial~\eqref{e:DS} and reduced-order~\eqref{e:dynamic system'} systems admits the estimate
\begin{multline*}
\bigl|d^H\exp_t(A)b-\hat d^H\exp_t(\widehat{A})\hat b\bigr|\le\\
\le
\max_{\substack{s\in[0,1]\\\mu\in\co\{\hat\mu_1,\dots,\hat\mu_{\hat m}\}}}\biggl|d^H\Bigl[
\Omega(A)[v(A)]^{-1}\frac{\bigl(v\exp_t\bigr)^{{(\hat n)}}
\bigl((1-s)\mu\mathbf1+sA\bigr)}{\hat n!}\Bigr]b\biggr|,
\end{multline*}
where
\begin{align*}
v(\lambda)&=\sum_{k=1}^{m}(\lambda-\lambda_k)^{\varkappa_k+\chi_k},\\
\Omega(\lambda)&=\sum_{k=1}^{\hat m}(\lambda-\hat\mu_k)^{\hat\omega_k}.
\end{align*}
\end{theorem}
We note that under assumptions of Theorem~\ref{t:est Arnoldi2}
\begin{equation*}
\hat n=\sum_{k=0}^{m}\varkappa_k+\chi_k=\sum_{k=1}^{\hat m}\hat w_k.
\end{equation*}

\begin{proof}
By Proposition~\ref{p:reduced order system},
\begin{equation*}
d^H\exp_t(A)\,b-\hat d^H\exp_t(\widehat{A})\,\hat b
=d^H\bigl(\exp_t(A)-r_t(A)\bigr)\,b,
\end{equation*}
where $r_t$ is a function of the form~\eqref{e:r_t} that interpolates
the function $\exp_t(\lambda)=e^{\lambda t}$ at the points $\hat\mu_1,\dots,\hat\mu_{\hat m}$ with multiplicities $\hat w_1,\dots,\hat w_{\hat m}$.

It remains to apply Corollary~\ref{c:Mathias:r5}, see also Theorem~\ref{t:Mathias:r1}. The degree $L$ of the numerator of function~\eqref{e:r_t} is less than or equal to $-1+\sum_{k=0}^m\varkappa_k+\chi_k$. Since the vectors~\eqref{e:Krylov vectors:V}, \eqref{e:Krylov vectors:Lambda}, \eqref{e:more Krylov vectors:V}, and~\eqref{e:more Krylov vectors:Lambda} form a basis, the order $\hat n$ of the matrix $\widehat{A}$ (this order determines the number of interpolation conditions) equals $\sum_{k=0}^m\varkappa_k+\chi_k$. Therefore the assumption $L\le\hat n-1$ from Theorem~\ref{t:Mathias:r1} is fulfilled. Furthermore, the denominator  $v(\lambda)=\prod_{k=1}^{m}(\lambda-\lambda_k)^{\varkappa_k+\chi_k}$ of function~\eqref{e:r_t}, by assumptions of Theorem~\ref{t:est Arnoldi2}, does not vanish both at the points of interpolation $\hat\mu_k$ and on $\sigma(A)$. Thus, all assumptions of Theorem~\ref{t:Mathias:r1} are fulfilled.
\end{proof}

\section{One-sided rational Arnoldi}\label{s:Arnoldi-1}
Theorem~\ref{t:est Arnoldi1} below is an analogue of Theorem~\ref{t:est Arnoldi2} for the approximation $t\mapsto Ve^{\widehat{A}t}\hat b$ of the vector impulse response $t\mapsto e^{At}b$. It corresponds to the one-sided Arnoldi method that allows one to calculate approximately the whole vector $e^{At}b$.

 \begin{proposition}[{\rm\cite[Theorem 3.3]{Guttel13}}]\label{p:Arnoldi-r1}
Let the image of the matrix $V$ contain vectors~\eqref{e:Krylov vectors:V} and~\eqref{e:more Krylov vectors:V}, and let $V^HV=\mathbf1_{\hat n\times\hat n}$.
We consider matrices~\eqref{e:hats} with $\Lambda=V^H$.
Let $\sigma(\widehat A)$ consists of the points $\hat\mu_1,\dots,\hat\mu_{\hat m}\in\mathbb C$, and let $\hat w_1,\dots,\hat w_{\hat m}$ be their algebraic multiplicities.

Let the points $\lambda_1$, \dots, $\lambda_m\in\mathbb C$ be not both in $\sigma(A)$ and in $\sigma(\widehat{A})$. Let $f$ be an analytic function defined on a neighborhood of the union of $\sigma(A)$ and $\sigma(\widehat{A})$.

Let a rational function $r_t$ of the form
 \begin{equation}\label{e:r_t-r1}
r_t(\lambda)=\sum_{k=1}^{m}\sum_{j=1}^{\varkappa_k}\frac{g_{jk}(t)}{{(\lambda_k-\lambda)^{j}}}+
\sum_{j=0}^{\varkappa_0-1}g_{j0}(t)\lambda^j
\end{equation}
satisfy the following interpolation assumptions{\rm:} the function
$r_t$ coincides with the function $\exp_t$ at all points $\hat\mu_1,\dots,\hat\mu_{\hat m}$ of $\sigma(\widehat A)$ with the derivatives up to the order $\hat w_1-1,\dots,\hat w_{\hat m}-1${\rm:}
\begin{equation*}
r_t^{(j)}(\hat\mu_k)=\exp_t^{(j)}(\hat\mu_k),\qquad k=1,2,\dots,\hat m;\;j=0,1,\dots,\hat w_{\hat m}-1.
\end{equation*}

Then one has
\begin{equation*}
r_t(A)\,b=V\exp_t(\widehat{A})\,\hat b.
\end{equation*}
 \end{proposition}

\begin{proof}
The proof is similar to that of Proposition~\ref{p:reduced order system}.
\end{proof}

\begin{theorem}\label{t:est Arnoldi1}
Let vectors~\eqref{e:Krylov vectors:V} and \eqref{e:more Krylov vectors:V} form a basis in the image of the matrix $V$, and $V^HV=\mathbf1_{\hat n\times\hat n}$.
Consider matrices~\eqref{e:hats} with $\Lambda=V^H$.
Let $\sigma(\widehat A)$ consist of the points $\hat\mu_1,\dots,\hat\mu_{\hat m}\in\mathbb C$, and let $\hat w_1,\dots,\hat w_{\hat m}$ be their algebraic multiplicities.

Let the points $\lambda_1$, \dots, $\lambda_m\in\mathbb C$ be not both in $\sigma(A)$ and in $\sigma(\widehat{A})$. Let $f$ be an analytic function defined on a neighbourhood of the union of $\sigma(A)$ and $\sigma(\widehat{A})$.
Then the difference between vector impulse responses of initial~\eqref{e:DS} and reduced-order~\eqref{e:dynamic system'} systems admits the estimate
\begin{multline}\label{e:est:one-sided}
\bigl\lVert\exp_t(A)b-V\exp_t(\widehat{A})\,\hat b\bigr\rVert\\
\le
\max_{\substack{s\in[0,1]\\\mu\in\co\{\hat\mu_1,\dots,\hat\mu_{\hat m}\}}}
\biggl\lVert\Omega(A)[v(A)]^{-1}\frac{\bigl(v\exp_t\bigr)^{{(\hat n)}}
\bigl((1-s)\mu\mathbf1+sA\bigr)}{\hat n!}b\biggr\rVert,
\end{multline}
where
\begin{align*}
v(\lambda)&=\sum_{k=1}^{m}(\lambda-\lambda_k)^{\varkappa_k},\\
\Omega(\lambda)&=\sum_{k=1}^{\hat m}(\lambda-\hat\mu_k)^{\hat\omega_k}.
\end{align*}
\end{theorem}
We note that under assumptions of Theorem~\ref{t:est Arnoldi1}
\begin{equation*}
\hat n=\sum_{k=0}^{m}\varkappa_k=\sum_{k=1}^{\hat m}\hat w_k.
\end{equation*}

\begin{proof}
The proof is similar to that of Theorem~\ref{t:est Arnoldi2}.
By Proposition~\ref{p:Arnoldi-r1},
\begin{equation*}
\exp_t(A)b-V\exp_t(\widehat{A})\hat b=\exp_t(A)b-r_t(A)\,b.
\end{equation*}
It remains to apply Corollary~\ref{c:Mathias:r6}.
\end{proof}

\section{Numerical range}\label{s:Numerical range}
In this section we describe (Examples~\ref{ex:1} and~\ref{ex:2}) two cases when estimate~\eqref{e:est:one-sided} can be used effectively.

The \emph{numerical range} of a matrix $A\in\mathbb C^{n\times n}$ is \cite{Gustafson-Rao} the set
\begin{equation*}
w(A)=\{\,\langle Az,z\rangle:\,\lVert z\rVert_2=1\,\}.
\end{equation*}
It is known~\cite[p.~4]{Gustafson-Rao} that $w(A)$ is a closed convex subset of $\mathbb C$.
The numerical range $w(A)$ of a normal matrix $A$ coincides~\cite[p.~16]{Gustafson-Rao} with the convex hall of~$\sigma(A)$.

\begin{proposition}\label{p:properties of num range}
The numerical range $w(A)$ possesses the following properties{\rm:}
\begin{itemize}
 \item[{\rm(a)}] $w(A)$ is a compact set{\rm;}
 \item[{\rm(b)}] $w(A)$ is contained in the ball of radius $\lVert A\rVert_{2\to2}$ centered at zero{\rm;}
 \item[{\rm(c)}] $w(A)$ contains $\sigma(A)${\rm;}
 \item[{\rm(d)}] $w(\alpha A)=\alpha w(A)$, $\alpha\in\mathbb C$.
\end{itemize}
\end{proposition}
\begin{proof}
Evident.
\end{proof}

\begin{proposition}\label{p:Kurgalin}
Let the columns of the matrix $V\in\mathbb C^{n\times\hat n}$ be orthonormalized and assump\-tion~\eqref{e:Lambda=VH} be fulfilled.
Then the numerical range $w(\widehat{A})$ and {\rm(}consequently{\rm)} the spectrum $\sigma(\widehat{A})$ of the matrix $\widehat{A}=\Lambda AV$ are contained in the numerical range $w(A)$ of the matrix $A$.
\end{proposition}
\begin{proof}
First, we notice that under the assumptions of the proposition $\lVert V\varphi\rVert=\lVert\varphi\rVert$ for any $\varphi\in\mathbb C^n$. In fact, by~\eqref{e:Labmda V=1} and Proposition~\ref{p:Lambda=VH},
\begin{equation*}
\lVert V\varphi\rVert=\sqrt{\langle V\varphi,V\varphi\rangle}
=\sqrt{\langle\Lambda V\varphi,\varphi\rangle}=\sqrt{\langle\varphi,\varphi\rangle}
=\lVert\varphi\rVert.
\end{equation*}

Let $\varphi\in\mathbb C^n$ be an arbitrary vector such that $\lVert\varphi\rVert=1$. Then
\begin{equation*}
\langle\widehat{A}\varphi,\varphi\rangle=\langle\Lambda AV\varphi,\varphi\rangle
=\langle AV\varphi,V\varphi\rangle\in w(A),
\end{equation*}
because $\lVert V\varphi\rVert=1$.
\end{proof}

\begin{example}\label{ex:1}
Let a matrix $A$ be self-adjoint and its spectrum be contained in a segment $[a,b]$. Hence, by Propositions~\ref{p:properties of num range} and~\ref{p:Kurgalin},  $\sigma(\widehat{A})\subseteq w(\widehat{A})\subseteq w(A)\subseteq[a,b]$. We recall that since a function $f(A)$ of a self-adjoint matrix $A$ is normal, the norm $\lVert f(A)\rVert$ coincides with the maximum of $|f(\lambda)|$ on the spectrum of $A$. Therefore the right-hand side of~\eqref{e:est:one-sided} can be estimated by
\begin{multline*}
\max_{\substack{s\in[0,1]\\\lambda,\mu\in[a,b]}}
\biggl|\Omega(\lambda)[v(\lambda)]^{-1}\frac{\bigl(v\exp_t\bigr)^{{(\hat n)}}
\bigl((1-s)\mu+s\lambda\bigr)}{\hat n!}\biggr|\cdot\lVert b\rVert\\
=\max_{\lambda\in[a,b]}
\biggl|\Omega(\lambda)[v(\lambda)]^{-1}\frac{\bigl(v\exp_t\bigr)^{{(\hat n)}}
(\lambda)}{\hat n!}\biggr|\cdot\lVert b\rVert.
\end{multline*}
\end{example}

We set~\cite[p. 11]{Dahlquist58},~\cite[Theorem 10.5]{Hairer-Norsett-Wanner:eng:2nd}, \cite{Lozinskii:eng} (clearly, the matrix $A+A^H$ is self-adjoint)
\begin{equation*}
\mu(A)=\max\Bigl\{\lambda:\,\lambda\in\sigma\Bigl(\frac{A+A^H}2\Bigr)\Bigr\}.
\end{equation*}
The number $\mu(A)$ is called the \emph{logarithmic norm} of $A$.

\begin{proposition}\label{p:num range and mu}
For any matrix $A\in\mathbb C^{n\times n}$ one has
\begin{equation*}
\mu(A)=\max\{\,\Real\lambda:\,\lambda\in w(A)\,\}.
\end{equation*}
\end{proposition}
\begin{proof}
Indeed,
\begin{multline*}
\max\{\,\Real\lambda:\,\lambda\in w(A)\,\}=\max\{\,\Real\langle Az,z\rangle:\,\lVert z\rVert_2=1\,\}=\\
=\max\biggl\{\,\Real\Bigl(\Bigl\langle\frac{A+A^H}2z,z\Bigr\rangle+\Bigl\langle\frac{A-A^H}2z,z\Bigr\rangle\Bigr):\,\lVert z\rVert_2=1\,\biggr\}=\\
=\max\biggl\{\,\Real\Bigl\langle\frac{A+A^H}2z,z\Bigr\rangle:\,\lVert z\rVert_2=1\,\biggr\}=\\
=\max\biggl\{\,\Bigl\langle\frac{A+A^H}2z,z\Bigr\rangle:\,\lVert z\rVert_2=1\,\biggr\}=\mu(A).\qed
\end{multline*}
\renewcommand\qed{}
\end{proof}

\begin{remark}\label{r:w(A)}
We recall~\cite[p.~137]{Gustafson-Rao} the algorithm for approximate calculation (more precisely, estimation from without) of the numerical range $w(A)$ of $A$. We denote by $q_{\max}(A)$ the largest eigenvalue of the (self-adjoint) matrix $\frac{A+A^H}2$, and we denote by $q_{\min}(A)$ the smallest eigenvalue of the matrix $\frac{A+A^H}2$.
We recall that $q_{\min}(A)$ and $q_{\max}(A)$ can be calculated by standard tools~\cite{Wolfram}.
By definition,
\begin{equation*}
\mu(A)=q_{\max}(A).
\end{equation*}

Hence, by Proposition~\ref{p:num range and mu},
\begin{equation*}
q_{\max}(A)=\max\{\,\Real\lambda:\,\lambda\in w(A)\,\}.
\end{equation*}
Therefore,
\begin{equation*}
w(A)\subseteq\{\,\lambda\in\mathbb C:\,\Real\lambda\le q_{\max}(A)\,\}.
\end{equation*}
Applying this inclusion to $-A$ we arrive at
\begin{equation*}
w(A)\subseteq\{\,\lambda\in\mathbb C:\,q_{\min}(A)\le\Real\lambda\le q_{\max}(A)\,\}.
\end{equation*}

Further, we take an arbitrary $\varphi\in\mathbb R$ and consider the matrix $A_\varphi=e^{-i\varphi}A$. By Proposition~\ref{p:properties of num range}(d),
\begin{equation*}
w(A)=e^{i\varphi}w(A_\varphi).
\end{equation*}
Therefore,
\begin{equation*}
w(A)\subseteq e^{i\varphi}\{\,\lambda\in\mathbb C:\,q_{\min}(A_\varphi)\le\Real\lambda\le q_{\max}(A_\varphi)\,\}.
\end{equation*}
Taking several $\varphi$, we construct the intersection of the corresponding strips that contain $w(A)$. In fact, already two angles, 0 and $-\pi/2$, give a rectangle that contains $w(A)$.
\end{remark}

The following theorem shows that an analytic function of a matrix can be effectively estimated via the values of the function on the numerical range.

\begin{theorem}[{see~\cite{Badea-Crouzeix-Delyon06,Beckermann-Crouzeix07,Crouzeix07} and references therein}]\label{t:Crouzeix}
Let a function $f$ be defined and analytic in a neighborhood of the numerical range $w(A)$ of a square matrix~$A$. Then
\begin{equation*}
\lVert f(A)\rVert\le C\,\max_{\lambda\in w(A)}|f(\lambda)|,
\end{equation*}
where $C=11.08$. If the neighborhood is an ellipse, then $C=3.16$. If the neighborhood is an disc, then $C=2$.
\end{theorem}

\begin{example}\label{ex:2}
We give another example when the right-hand side of~\eqref{e:est:one-sided} can be estimated effectively. Let the numerical range $w(A)$ be contained in a closed convex subset $\Psi\subseteq\mathbb C$. As the simplest examples, one can take for $\Psi$ the ball of radius $\lVert A\rVert$ centered at zero. Or one can take for $\Psi$ (according to Remark~\ref{r:w(A)}) the rectangle $[q_{\min}(A),q_{\max}(A)]\times[iq_{\min}(-iA),iq_{\max}(-iA)]$. By Theorem~\ref{t:Crouzeix}, the right-hand side of~\eqref{e:est:one-sided} can be estimated by
\begin{multline*}
11.08\cdot\max_{\substack{s\in[0,1]\\\lambda,\mu\in\Psi}}
\biggl|\Omega(\lambda)[v(\lambda)]^{-1}\frac{\bigl(v\exp_t\bigr)^{{(\hat n)}}
\bigl((1-s)\mu+s\lambda\bigr)}{\hat n!}\biggr|\cdot\lVert b\rVert\\
=11.08\cdot\max_{\lambda\in\Psi}
\biggl|\Omega(\lambda)[v(\lambda)]^{-1}\frac{\bigl(v\exp_t\bigr)^{{(\hat n)}}
(\lambda)}{\hat n!}\biggr|\cdot\lVert b\rVert.
\end{multline*}
\end{example}

\section{Numerical experiment}\label{s:experiment}
In this section, we present a numerical experiment that shows the gap between the left-hand and right-hand sides of~\eqref{e:est:one-sided}. We carry out our numerical experiments using `Mathematica'~\cite{Wolfram}.

For $f$ we take the function $f(\lambda)=e^\lambda$, i.e. $f=\exp_t$ with $t=1$.
We consider matrices $A$ with spectrum lying in the rectangle $[-1,0]+[-i\pi,i\pi]$.
We use the Euclidian norm $\lVert\cdot\rVert_2$ for vectors from $\mathbb C^n$.

We points $\lambda_k$, $k=1,\dots,8$, are determined by the rectangle $[-1,0]+[-i\pi,i\pi]$ in the following way. We take 18 points $0$, $\pm i\pi/4$, $\pm i\pi/2$, $\pm i3\pi/4$, $\pm i\pi$, and $-1$, $-1\pm i\pi/4$, $-1\pm i\pi/2$, $-1\pm i3\pi/4$, $-1\pm i\pi$ on the boundary of this rectangle. On the left Fig.~\ref{f:cir}, these points are marked by medium black dots. Then we calculate (by formulae from~\cite{Baker-Graves-Morris96:eng}) a rational function $q$ of degree $[9/8]$ that interpolates the function $f(\lambda)=e^\lambda$ at these 18 points. We take the poles $\lambda_k$, $k=1,\dots,8$, of the function $q$ as the zeroes of the function $v$ from~\eqref{e:est:one-sided}; thus, implicitly, $\lambda_k$, $k=1,\dots,8$, are the poles of the function $r_t$ from~\eqref{e:r_t-r1}. On the left Fig.~\ref{f:cir}, these points are marked by the sign $\oplus$.

We put $N=1024$. We take complex numbers $\nu_i$, $i=1,\dots,N$, uniformly distributed in the rectangle $[-1,0]+[-i\pi,i\pi]$. We consider the diagonal matrix $D$ of the size $N\times N$ with the diagonal entries $\nu_i$. We create a matrix $S$, whose entries are random numbers uniformly distributed in $[-1,1]+[-i,i]$. Then, we consider the matrix $A=SDS^{-1}$. Clearly, $\sigma(A)$ consists of the numbers $\nu_i$. We interpret $A$ as a random matrix whose spectrum is contained in the rectangle $[-1,0]+[-i\pi,i\pi]$.
On the right Fig.~\ref{f:cir}, we show an example of the spectrum of such a matrix.

We calculate the exact matrix $e^A$ by the formula
\begin{equation*}
e^A=SES^{-1},
\end{equation*}
where $E$ is the diagonal matrix with the diagonal entries $e^{\nu_i}$.

We take a random vector $b\in\mathbb C^{1024}$ with $\lVert v\rVert_2=1$.
We construct the matrix $V\in\mathbb C^{1024\times9}$ with orthonormal columns whose image coincides with the linear span of the vectors
\begin{equation*}
b,\;(\lambda_{1}I-A)^{-1}b,\;\dots,\;(\lambda_{8}I-A)^{-1}b.
\end{equation*}
We put $\Lambda=V^H$, consider $\widehat{A}=V^HAV\in\mathbb C^{n\times n}$, and calculate (by a standard tool) the spectrum $\sigma(\widehat{A})=\{\hat\mu_1,\dots,\hat\mu_{9}\}$ of the matrix $\widehat{A}$. On the right Fig.~\ref{f:cir}, the points $\hat\mu_k$ are marked by large black dots.

Then we calculate $e^{\widehat{A}}$ (again by a standard tool). Next we calculate the left-hand size of~\eqref{e:est:one-sided} (and denote it by $e_0$):
\begin{equation*}
e_0=\bigl\lVert f(A)b-Vf(\widehat{A})\,\hat b\bigr\rVert
=\bigl\lVert e^Ab-Ve^{\widehat{A}}\,\hat b\bigr\rVert_2.
\end{equation*}

We draw the boundary of the convex hall of $\sigma(\widehat{A})$; it is a broken line shown in the right Fig.~\ref{f:cir}. According to the Maximum modulus principle for analytic functions, we replace the maximum over $\mu\in\co\{\hat\mu_1,\dots,\hat\mu_{9}\}$ by the maximum over the boundary.

We calculate $\Omega(A)[v(A)]^{-1}\frac{(vf)^{{(9)}}
((1-s)\mu\mathbf1+sA)}{9!}b$ by the rule
\begin{equation*}
\Omega(A)[v(A)]^{-1}\frac{\bigl(vf\bigr)^{{(9)}}
\bigl((1-s)\mu\mathbf1+sA\bigr)}{9!}b=SHS^{-1}b,
\end{equation*}
where $H$ is a diagonal matrix with the diagonal entries
\begin{equation*}
h_{i}=\Omega(\nu_i)[v(\nu_i)]^{-1}\frac{\bigl(vf\bigr)^{{(9)}}
\bigl((1-s)\mu\mathbf1+s\nu_i\bigr)}{9!}.
\end{equation*}
After that, we calculate
\begin{equation*}
\biggl\lVert\Omega(A)[v(A)]^{-1}\frac{\bigl(vf\bigr)^{{(9)}}
\bigl((1-s)\mu\mathbf1+sA\bigr)}{9!}b\biggr\rVert_{2\to2}
\end{equation*}
for a discrete family of $\mu$'s and $s$'s. More precisely, we mark approximately 50 uniformly distributed points on the boundary; we denote them by $\mu_k$ (they are marked at the right-hand side of~\eqref{e:est:one-sided} by small black stars). Next, we take 11 points $s_l=l/10$, $l=0,\dots,10$, in the segment $[0,1]$. We take for $\mu$ only the points $\mu_k$, and we take for $s$ only the points $s_l$. Finally, we take the maximum over all the points. Thus, we obtain the right-hand side of~\eqref{e:est:one-sided}. We denote it by~$e_1$.

We repeated the described experiment 100 times. After each repetition, we saved 3 numbers: the value $e_0$ of the left-hand size of~\eqref{e:est:one-sided}, the value $e_1$ of the right-hand size, and the ratio $e_1/e_0$. Then we calculated the average values. They are as follows:
the mean value of  $e_0$ is $8.2\cdot10^{-7}$ with the standard deviation $8.1\cdot10^{-7}$,
the mean value of  $e_1$ is $2.9\cdot10^{-6}$ with the standard deviation $1.1\cdot10^{-5}$,
the mean value of  $e_1/e_0$ is $1.8$ with the standard deviation $2.1$.

The mean value $1.8$ of $e_1/e_0$ shows that the estimate is rather close to the real accuracy. 
\begin{figure}[htb]
\begin{center}
\includegraphics[width=\textwidth]{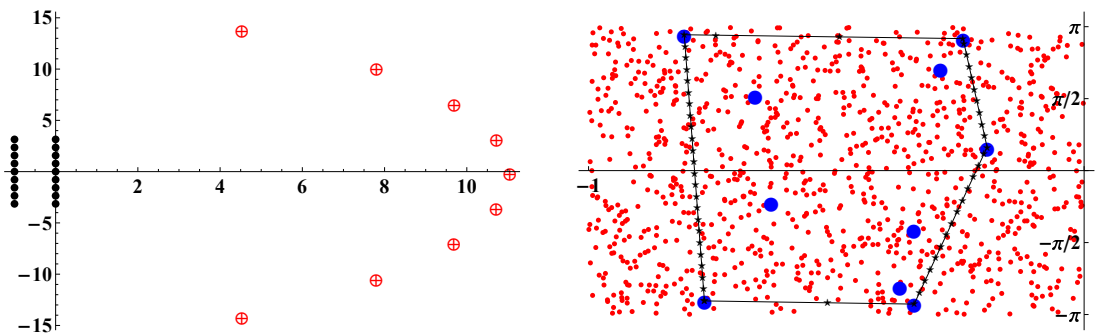}
\caption{Left: the poles $\lambda_k$, $k=1,\dots,9$; right: the spectra of $A$ and $\widehat{A}$}\label{f:cir}
\end{center}
\end{figure}

\section*{Acknowledgements}\label{s:Acknowledgements}
The reported study was funded by Russian Foundation for Basic Research and Czech Science Foundation according to the research projects No.~19-52-26006 and No.~20-10591J.
We also acknowledge the support by ERDF/ESF ``Centre of Advanced Applied Sciences'' (No. CZ.02.1.01/0.0/0.0/16\textunderscore019/0000778).

\providecommand{\bysame}{\leavevmode\hbox to3em{\hrulefill}\thinspace}
\providecommand{\MR}{\relax\ifhmode\unskip\space\fi MR }
\providecommand{\MRhref}[2]{%
  \href{http://www.ams.org/mathscinet-getitem?mr=#1}{#2}
}
\providecommand{\href}[2]{#2}

\end{document}